\documentclass[14pt]{amsart}

\usepackage{cases}
\usepackage{amsmath}
\usepackage{amsfonts}
\usepackage{bm}
\usepackage{amsfonts,amsmath,amssymb,amscd,bbm,amsthm,mathrsfs,dsfont}
\usepackage{mathrsfs}
\usepackage{pb-diagram}
\usepackage{amssymb}
\usepackage{xypic}
\usepackage[dvips]{graphicx}
\usepackage[all]{xy}
\usepackage{CJK}
\usepackage[dvipsnames]{xcolor}

\newtheorem{theorem}{Theorem}[section]
\newtheorem{lemma}[theorem]{Lemma}
\newtheorem{conjecture}[theorem]{Conjecture}
\newtheorem{corollary}[theorem]{Corollary}
\newtheorem{proposition}[theorem]{Proposition}
\theoremstyle{definition}
\newtheorem{definition}[theorem]{Definition}

\theoremstyle{remark}
\newtheorem{remark}[theorem]{Remark}
\date{version of \today}

\newcommand{\huang}[1]{\textcolor{green}{Huang: #1}}

 \normalbaselines
\title[Categorification of cluster algebras and conjectures on $\bf{g}$-vectors]{Categorification of sign-skew-symmetric cluster algebras and some conjectures on $\bf{g}$-vectors}

\author{Peigen Cao $\;\;\;\;\;\;$ Min Huang $\;\;\;\;\;\;$ Fang Li $\;\;\;\;\;\;$}
\address{Peigen Cao
\newline Department
of Mathematics, Zhejiang University (Yuquan Campus), Hangzhou, Zhejiang
310027,  P.R.China}
\email{peigencao@126.com}

\address{Min Huang
\newline Department
of Mathematics, Zhejiang University (Yuquan Campus), Hangzhou, Zhejiang
310027,  P.R.China}
\email{minhuang1989@hotmail.com}

\address{Fang Li
\newline Department
of Mathematics, Zhejiang University (Yuquan Campus), Hangzhou, Zhejiang
310027, P.R.China}
\email{fangli@zju.edu.cn}

\begin{document}

\begin{abstract}
Using the unfolding method given in \cite{HL}, we prove the conjectures on sign-coherence and a recurrence formula respectively of ${\bf g}$-vectors for acyclic sign-skew-symmetric cluster algebras. As a following consequence, the conjecture is affirmed in the same case which states that the ${\bf g}$-vectors of any cluster form a basis of $\mathbb Z^n$. Also, the additive categorification of an acyclic sign-skew-symmetric cluster algebra $\mathcal A(\Sigma)$ is given, which is realized as $(\mathcal C^{\widetilde Q},\Gamma)$ for a Frobenius $2$-Calabi-Yau category  $\mathcal C^{\widetilde Q}$ constructed from an unfolding $(Q,\Gamma)$ of the acyclic exchange matrix $B$ of $\mathcal A(\Sigma)$.

\end{abstract}

\maketitle

\section{Introduction }

Cluster category was introduced by \cite{BMRRT} for an acyclic quiver. In general, we view a Hom-finite $2$-Calabi-Yau triangulated category $\mathcal C$ which has a cluster structure as a cluster category, see \cite{BIRTS}. In fact, the mutation of a cluster-tilting object $T$ in $\mathcal C$ categorifies the mutation of a quiver $Q$, where the quiver $Q$ is the Gabriel quiver of the algebra $End_{\mathcal C}(T)$. Cluster character gives an explicit correspondence between certain cluster objects of $\mathcal C$ and all the clusters of $\mathcal A(\Sigma(Q))$, where $\Sigma(Q)$ means the seed associated with $Q$. For details, see \cite{FK}, \cite{PP} and \cite{PP1}. Thus, cluster category and cluster character are useful tools to study a cluster algebra.

Let $\mathcal A(\Sigma)$ be a cluster algebra with principal coefficients at $\Sigma=(X,Y,B)$, where $B$ is an $n\times n$ sign-skew-symmetric integer matrix, $Y=(y_1,\cdots,y_n)$. The celebrate Laurent phenomenon says that $\mathcal A$ is a subalgebra of $\mathbb Z[y_{n+1},\cdots,y_{2n}][X^{\pm 1}]$. Setting $deg(x_i)=e_i$, $deg(y_j)=-b_j$, then $\mathbb Z[y_{n+1},\cdots,y_{2n}][X^{\pm 1}]$ becomes to a graded algebra, where $\{e_{i}\;|\;i=1,\cdots,n\}$ is the standard basis of $\mathbb Z^n$ and $b_{i}$ is the $i$-th column of $B$. Under such $\mathbb Z^n$-grading, the cluster algebra $\mathcal A$ is a graded subalgebra, in which the degree of a homogenous element is called its {\bf ${\bf g}$-vector}; furthermore, each cluster variable $x$ in $\mathcal A(\Sigma)$ is homogenous,  denoted its ${\bf g}$-vector as ${\bf g}(x)$. For details, see \cite{fz4}.

It was conjectured that

\begin{conjecture}\label{$g$-vec}(\cite{fz4}, Conjecture 6.13)
For any cluster $X'$ of $\mathcal A(\Sigma)$ and all $x\in X'$, the vectors ${\bf g}(x)$ are sign-coherent, which means that the $i$-th coordinates of all these vectors are either all non-negative or all non-positive.
\end{conjecture}

Such conjecture has been proved in the skew-symmetrizable case in (\cite{GHKK}, Theorem 5.11).

\begin{conjecture}\label{basis}(\cite{fz4}, Conjecture 7.10(2))
For any cluster $X'$ of $\mathcal A(\Sigma)$, the vectors ${\bf g}(x),x\in X'$ form a $\mathbb Z$-basis of the lattice $\mathbb Z^n$.
\end{conjecture}

In terms of cluster pattern, fixed a regular tree $\mathbb T_n$, for any vertices $t$ and $t_0$ of $\mathbb T_n$, let ${\bf g}^{B^0; t_0}_{1; t},\cdots, {\bf g}^{B^0; t_0}_{n; t}$ denote the ${\bf g}$-vectors of the cluster variables in the seed $\Sigma_t$ with respect to the principal coefficients seed $\Sigma_{t_0}=(X_0,Y,B_0)$. When set different vertices to be principal coefficients seeds, it is conjectured that the ${\bf g}$-vectors with respect to a fixed vertex $t$ of $\mathbb T_n$ have the following relation.

\begin{conjecture}\label{basechange}(\cite{fz4}, Conjecture 7.12)
Let $t_1\overset{k}{--} t_2\in \mathbb T_n$ and let $B^2=\mu_k(B^1)$. For $a\in [1,n]$ and $t\in \mathbb T_n$, assume ${\bf g}^{B^1; t_1}_{a; t}=(g_{1}^{t_1},\cdots,g_{n}^{t_1})$ and ${\bf g}^{B^2; t_2}_{a; t}=(g_{1}^{t_2},\cdots,g_{n}^{t_2})$, then
\begin{equation}\label{eqn3.1}
g_i^{t_2}=\begin{cases}-g_k^{t_1}& \text{if }i=k;\\g_i^{t_1}+[b_{ik}^{t_1}]_+g_k^{t_1}-b_{ik}^{t_1}min(g_k^{t_1},0)&\text{if }i\neq k.\end{cases}
\end{equation}
\end{conjecture}

\begin{remark}\label{remark}
(1)~As is said in Remark 7.14 of \cite{fz4}, it is easy to see that Conjectures \ref{$g$-vec} and \ref{basechange} imply Conjecture \ref{basis}.

(2)~In the skew-symmetrizable case, the sign-coherence of the ${\bf c}$-vectors can deduce Conjectures \ref{$g$-vec} and \ref{basechange}, see \cite{NZ}, where the given method strongly depends on the skew-symmetrizability. So far, the similar conclusion have not been given in the sign-skew-symmetric case. Further,  the sign-coherence of the ${\bf c}$-vectors has been proved in \cite{HL} for the acyclic sign-skew-symmetric case, because this is equivalent to $F$-polynomial has constant term $1$, see \cite{fz4}. Therefore, for the acyclic sign-skew-symmetric case, it is interesting to study directly Conjectures \ref{$g$-vec} and \ref{basechange}.
\end{remark}

Unfolding of skew-symmetrizable matrices is introduced by Zelevinsky, whose aim is to characterize skew-symmetrizable cluster algebras using the version in skew-symmetric case. The second and third authors of this paper improved in \cite{HL} such method to arbitrary sign-skew-symmetric matrices. According to this previous work, there is an unfolding $(\widetilde Q, F, \Gamma)$ of any acyclic $m\times n$ matrix $\widetilde B$, and also a $2$-Calabi-Yau Frobenius category $\mathcal C^{\widetilde Q}$ with $\Gamma$ action constructed.

Our motivation and the main results of this paper are two-fold.

(1)~Give the $\Gamma$-equivariant cluster character for $\underline {\mathcal C^Q}$, which can be
regarded as the additive categorification of the cluster algebra $\mathcal A(\Sigma(Q))$. See Theorem \ref{char} and Theorem \ref{reach}.

(2)~Solve Conjectures \ref{$g$-vec} and \ref{basechange} in the acyclic case. See Theorem \ref{sign-c} and Theorem \ref{base}. As a consequence, in the same case,  Conjecture \ref{basis} follows to be affirmed.

\section{An overview of unfolding method}\label{2}

In this section, we give a brief introduction of the concept of unfolding of totally sign-skew-symmetric cluster algebras and some necessary results in \cite{HL}.

For any sign-skew-symmetric matrix $B\in Mat_{n\times n}(\mathbb Z)$, one defines a quiver $\Delta(B)$ as follows: the vertices are $1,\cdots, n$ and there is an arrow from $i$ to $j$ if and only if $b_{ij}>0$. $B$ is called {\bf acyclic} if $\Delta(B)$ is acyclic, a cluster algebra is called {\bf acyclic} if it has an acyclic exchange matrix, see \cite{fz3}.

A locally finite {\bf ice quiver} is a pair $(Q,F)$ where $Q$ is a locally finite quiver without 2-cycles or loops and $F\subseteq Q_0$ is a subset of vertices called {\bf frozen vertices} such that there are no arrows among vertices of $F$. For a locally finite ice quiver $(Q,F)$, we can associate an (infinite) skew-symmetric row and column finite (i.e. having at most finite nonzero entries in each row and column) matrix $(b_{ij})_{i\in Q_0,j\in Q_0\setminus F}$, where $b_{ij}$ equals to the number of arrows from $i$ to $j$ minus the number of arrows from $j$ to $i$. In case of no confusion, for convenience, we also denote the ice quiver $(Q, F)$ as $(b_{ij})_{i\in Q_0, j\in Q_0\setminus F}$.

We say that {\em an ice quiver $(Q,F)$ admits the action of a group $\Gamma$} if $\Gamma$ acts on $Q$ such that $F$ is stable under the action. Let $(Q,F)$ be a locally finite ice quiver with an action of a group $\Gamma$ (maybe infinite). For a vertex $i\in Q_0\setminus F$, a {\bf $\Gamma$-loop} at $i$ is an arrow from $i$ to $h\cdot i$ for some $h\in \Gamma$, a {\bf $\Gamma$-$2$-cycle} at $i$ is a pair of arrows $i\rightarrow j$ and $j\rightarrow h\cdot i$ for some $j\notin \{h'\cdot i\;|\;h'\in \Gamma\}$ and $h\in \Gamma$. Denote by $[i]$ the orbit set of $i$ under the action of $\Gamma$. Say that {\bf $(Q, F)$ has no $\Gamma$-loops} (respectively, {\bf $\Gamma$-$2$-cycles}) {\bf at} $[i]$ if $(Q,F)$ has no $\Gamma$-loops ($\Gamma$-$2$-cycles, respectively) at any $i'\in [i]$.

\begin{definition}\label{orbitmu}(Definition 2.1, \cite{HL})
Let $(Q,F)=(b_{ij})$ be a locally finite ice quiver with an group $\Gamma$ action. Denote $[i]=\{h\cdot i\;|\;h\in \Gamma\}$ the orbit of vertex $i\in Q_0\setminus F$. Assume that $(Q, F)$ admits no $\Gamma$-loops and no $\Gamma$-$2$-cycles at $[i]$, we define an {\bf adjacent ice quiver}  $(Q', F)=(b'_{i'j'})_{i'\in Q_0, j'\in Q_0\setminus F}$ from $(Q,F)$ to be the quiver by following:

(1)~ The vertices are the same as $Q$,

(2)~ The arrows are defined as
 \[\begin{array}{ccl} b'_{jk} &=&
         \left\{\begin{array}{ll}
             -b_{jk}, &\mbox{if $j\in [i]$ or $k\in [i]$}, \\
              b_{jk}+\sum\limits_{i'\in[i]}\frac{|b_{ji'}|b_{i'k}+b_{ji'}|b_{i'k}|}{2}, &\mbox{otherwise.}
         \end{array}\right.
 \end{array}\]
Denote $(Q',F)$ as $\widetilde\mu_{[i]}((Q,F))$ and call $\widetilde\mu_{[i]}$  the {\bf orbit mutation} at direction $[i]$ or at $i$ under the action $\Gamma$. In this case, we say that $(Q,F)$ can {\bf do orbit mutation at} $[i]$.
\end{definition}

Note that if $\Gamma$ is the trivial group $\{e\}$, then the definition of orbit mutation of a quiver is the same as that of quiver mutation (see \cite{fz1}\cite{fz2}).

\begin{definition}\label{unfolding}(Definition 2.4, \cite{HL}) (i)~
For a locally finite ice quiver $(Q,F)=(b_{ij})_{i\in Q_0,j\in Q_0\setminus F}$ with a group $\Gamma$ (maybe infinite) action, let $\overline {Q_0}$ (respectively, $\overline F$) be the orbit sets of the vertex set $Q_0$ (respectively, the frozen vertex set $F$) under the $\Gamma$-action. Assume that $m=|\overline {Q_0}|<+\infty$, $m-n=|\overline{F}|$ and $Q$ has no $\Gamma$-loops and no $\Gamma$-$2$-cycles.

Define a sign-skew-symmetric matrix $B(Q,F)=(b_{[i][j]})_{[i]\in \overline{Q_0},[j]\in \overline{Q_0}\setminus \overline{F}}$ to $(Q,F)$ satisfying (1)~ the size of the matrix $B(Q,F)$ is $m\times n$; (2)~ $b_{[i][j]}=\sum\limits_{i'\in [i]}b_{i'j}$ for $[i]\in \overline {Q_0}$, $[j]\in \overline{Q_0}\setminus\overline{F}$.

(ii)~ For an $m\times n$ sign-skew-symmetric matrix $B$, if there is a locally finite ice quiver $(Q,F)$ with a group $\Gamma$ such that $B=B(Q,F)$ as constructed in (i), then we call $(Q,F,\Gamma)$  a {\bf covering} of $B$.

(iii)~ For an $m\times n$ sign-skew-symmetric matrix $B$, if there is a locally finite quiver $(Q,F)$ with an action of group $\Gamma$ such that $(Q,F,\Gamma)$  is a covering of $B$ and $(Q,F)$ can do arbitrary steps of orbit mutations, then $(Q,F,\Gamma)$ is called an {\bf unfolding} of $B$; or equivalently, $B$ is called the {\bf folding} of $(Q,F,\Gamma)$.
\end{definition}

\begin{remark}
The definition of unfolding is slight different with that in \cite{HL} where the definition was just applied to square matrices.
\end{remark}

By Lemma 2.5 of \cite{HL}, we have the following consequence.

\begin{lemma}\label{mut}
If $(Q, F, \Gamma)$ is an unfolding of $B$, for any sequence $[i_1],\cdots,[i_s]$ of orbits of $Q_0\setminus F$ under the action of $\Gamma$, then $(\widetilde \mu_{[i_s]}\cdots \widetilde \mu_{[i_s]}(Q, F), \Gamma)$ is a covering of $\mu_{[i_s]}\cdots\mu_{[i_1]}B$.
\end{lemma}

By Theorem 2.16 of \cite{HL}, we have

\begin{theorem}\label{mainlemma}
If $\widetilde B\in Mat_{m\times n}(\mathbb Z)$ ($m\geq n$) is an acyclic sign-skew-symmetric matrix, then $\widetilde B$ has an unfolding $(\widetilde Q,F,\Gamma)$, where $\widetilde Q$ is given using of Construction 2.6 in \cite{HL}.
\end{theorem}

\begin{proof}
Assume $\widetilde B=\left(\begin{array}{c}
B  \\
B'
\end{array}\right)$ with $B\in Mat_{n\times n}(\mathbb Z)$. Denote $\widetilde B'=\left(\begin{array}{cc}
B & -{B'}^{T}  \\
B' &  0
\end{array}\right)$. Since $\widetilde B$ is acyclic, $\widetilde B'\in Mat_{m\times m}(\mathbb Z)$ is acyclic. According to Construction 2.6 and Theorem 2.16 of \cite{HL}, $\widetilde B'$ has an unfolding $(\widetilde Q,\Gamma)$. Let $F\subset \widetilde Q_0$ be the vertices of $\widetilde Q_0$ corresponding to $B'$. Thus, it is clear that $(\widetilde Q,F,\Gamma)$ is an unfolding of $\widetilde B$.
\end{proof}

\begin{remark}
In \cite{HL}, we proved that $\widetilde Q$ is {\bf strongly almost finite}, that is,  $\widetilde Q$ is locally finite and has no path of  infinite length.
\end{remark}

This theorem means that an acyclic matrix is always totally sign-skew-symmetric. Thus, we can define a cluster algebra via an acyclic  matrix.

For an  acyclic matrix $\widetilde B\in Mat_{m\times n}(\mathbb Z)$ ($m\geq n$), assume $(\widetilde Q,F, \Gamma)$ is an unfolding of $\widetilde B$. Denote $\overline{\widetilde Q_0}$ and $\overline F$ be the orbits sets of vertices in $\widetilde Q_0$ and $F$. Let $\widetilde\Sigma=\Sigma(\widetilde Q)=(\widetilde X, \widetilde Y, \widetilde Q)$ be the seed associated with $(\widetilde Q, F)$, where $\widetilde X=\{x_{u}\;|\;u\in \widetilde Q_0\setminus F\}$, $\widetilde Y=\{y_v\;|\;v\in F\}$. Let $\Sigma=\Sigma(\widetilde B)=(X, Y,\widetilde B)$ be the seed associated with $\widetilde B$, where $X=\{x_{[i]}\;|\;[i]\in \overline{\widetilde Q_0}\setminus \overline F\}$, $Y=\{y_{[j]}\;|\;[j]\in \overline{F}\}$. It is clear that there is a surjective algebra homomorphism:
\begin{equation}\label{pi}
\pi:\;\; \mathbb Q[x^{\pm 1}_i, y_j\;|\;i\in \widetilde Q_0\setminus F, j\in F]\rightarrow \mathbb Q[x^{\pm 1}_{[i]},y_{[j]}\;|\;[i]\in \overline{\widetilde Q_0}\setminus \overline F, [j]\in \overline{F}]
\end{equation}
such that  $\pi(x_{i})= x_{[i]}$ and $\pi(y_j)= y_{[j]}$.

For any cluster variable $x_u\in \widetilde X$, define $\widetilde\mu_{[i]}(x_u)=\mu_u(x_u)$ if $u\in [i]$; otherwise,  $\widetilde\mu_{[i]}(x_u)=x_u$ if $u\not\in [i]$. Formally, write $\widetilde\mu_{[i]}(\widetilde X)=\{\widetilde\mu_{[i]}(x)\;|\;x\in \widetilde X\}$ and $\widetilde\mu_{[i]}({\widetilde X}^{\pm 1})=\{\widetilde\mu_{[i]}(x)^{\pm 1}\;|\;x\in \widetilde X\}$.

\begin{lemma}\label{mutationorbit}(Lemma 7.1, \cite{HL})
Keep the notations as above. Assume that $B$ is acyclic. If $[i]$ is an orbit of vertices with $i\in \widetilde Q_0\setminus F$, then

(1)~$\widetilde\mu_{[i]}(x_j)$ is a cluster variable of $\mathcal A(\widetilde Q)$ for any $j\in \widetilde Q_0\setminus F$,

(2)~$\widetilde\mu_{[i]}(\widetilde X)$ is algebraic independent over $\mathbb Q[y_j\;|\;j\in F]$.
\end{lemma}

By Lemma \ref{mutationorbit}, $\widetilde\mu_{[i]}(\widetilde \Sigma):=(\widetilde\mu_{[i]}(\widetilde X), \widetilde Y, \widetilde\mu_{[i]}(\widetilde Q))$ is a seed. Thus, we can define $\widetilde\mu_{[i_s]}\widetilde\mu_{[i_{s-1}]}\cdots\widetilde\mu_{[i_1]}(x)$ and $\widetilde\mu_{[i_s]}\widetilde\mu_{[i_{s-1}]}\cdots\widetilde\mu_{[i_1]}(\widetilde X)$ and $\widetilde\mu_{[i_s]}\widetilde\mu_{[i_{s-1}]}\cdots\widetilde\mu_{[i_1]}(\widetilde \Sigma)$ for any sequence $([i_1],[i_2],\cdots,[i_s])$ of orbits in $Q_0$.

\begin{theorem}\label{sur}(Theorem 7.5, \cite{HL})
Keep the notations as above with an acyclic sign-skew-symmetric matrix $B$ and $\pi$ as defined in (\ref{pi}).
 Restricting $\pi$ to $\mathcal A(\widetilde\Sigma)$, then $\pi:\mathcal A(\widetilde\Sigma)\rightarrow \mathcal A(\Sigma)$ is a surjective algebra morphism satisfying that $\pi(\widetilde\mu_{[j_k]}\cdots\widetilde\mu_{[j_1]}(x_{a} ))=\mu_{[j_k]}\cdots\mu_{[j_1]}(x_{[i]})\in \mathcal A(\Sigma)$ and $\pi(\widetilde\mu_{[j_k]}\cdots\widetilde\mu_{[j_1]}(\widetilde X))=\mu_{[j_k]}\cdots\mu_{[j_1]}(X)$ for any sequences of orbits $[j_1],\cdots,[j_k]$ and any $a\in [i]$.
\end{theorem}

In case $\mathcal A(\Sigma)$ with principal coefficients, from Lemma \ref{mut}, we may assume that $\mathcal A(\widetilde\Sigma)$ is also with principal coefficients. Let $\lambda:\bigoplus\limits_{i\in \widetilde Q_0\setminus F}\mathbb Z e_i\rightarrow \bigoplus\limits_{[i]\in \overline{\widetilde Q_0}\setminus \overline F}\mathbb Z e_{[i]}, e_i\rightarrow e_{[i]}$ be the group homomorphism, where $\bigoplus\limits_{i\in \widetilde Q_0\setminus F}\mathbb Z e_i$ (resp. $\bigoplus\limits_{[i]\in \overline {\widetilde Q_0}\setminus \overline F}\mathbb Z e_{[i]}$) is the free abelian group generated by $\{e_i\;|\;i\in \widetilde Q_0\setminus F\}$ (resp. $\{e_{[i]}\;|\;[i]\in \overline{\widetilde Q_0}\setminus \overline F\}$). Under such group homomorphism, $\mathcal A(\widetilde\Sigma)$ becomes a $\mathbb Z^n$-graded algebra such any cluster variable $x$ is homogenous with degree $\lambda({\bf g}(x))$.

\begin{theorem}\label{ho}
Keep the notations as in Theorem \ref{sur}. If $\mathcal A(\Sigma)$ with principal coefficients, then the restriction of $\pi$ to $\mathcal A(\widetilde\Sigma)$ is a $\mathbb Z^n$-graded surjective homomorphism.
\end{theorem}

\begin{proof}
Since $\lambda({\bf g}(x_i))=e_{[i]}={\bf g}(x_{[i]})$ and $\lambda({\bf g}(y_j))=-b_{[j]}={\bf g}(y_{[j]})$, where $b_{[j]}$ is the $[j]$-th column of $B$.
Further, because $\{x^{\pm 1}_i, y_j\;|\;i\in \widetilde Q_0\setminus F, j\in \widetilde Q_0\}$ is a generator of $\mathbb Q[x^{\pm 1}_i, y_j\;|\;i\in \widetilde Q_0\setminus F, j\in \widetilde Q_0]$, thus $\pi$ is homogenous. Then our result follows by Theorem \ref{sur}.
\end{proof}

\section{Cluster character in sign-skew-symmetric case}\label{3}

Let $\widetilde B\in Mat_{m\times n}(\mathbb Z)$ be an acyclic sign-skew-symmetric matrix, $(\widetilde Q, F, \Gamma)$ be an unfolding of $\widetilde B$ given in Theorem \ref{mainlemma}. Denote $\widetilde \Sigma=(\widetilde X,\widetilde Y,\widetilde Q)$ and $\Sigma=(X,Y,\widetilde B)$ be the seeds corresponding to $(\widetilde Q, F)$ and $\widetilde B$.

From $(\widetilde Q, F,\Gamma)$, we constructed a $2$-Calabi-Yau Frobenius category $\mathcal C^{\widetilde Q}$ in \cite{HL} such that $\Gamma$ acts on it exactly, i.e. each $h\in \Gamma$ acts on $\mathcal C^{\widetilde Q}$ as an exact functor. Furthermore, there exists a cluster tilting subcategory $\mathcal T_0$ of $\mathcal C^{\widetilde Q}$ such that the Gabriel quiver of $\underline {\mathcal T_0}$ is isomorphic to $\widetilde Q$, where $\underline {\mathcal T_0}$ is the subcategory of the stable category $\underline{\mathcal C^{\widetilde Q}}$ corresponding to $\mathcal T_0$. For details, see Lemma 4.15 of \cite{HL}. Since $\mathcal C^{\widetilde Q}$ is a $Hom$-finite $2$-Calabi-Yau Frobenius category,  $\underline{\mathcal C^{\widetilde Q}}$ follows to be a $Hom$-finite $2$-Calabi-Yau triangulated category. Write $[1]$ as the shift functor in $\underline{\mathcal C^{\widetilde Q}}$. For any object $X$ and subcategory $\mathcal X$ of $\mathcal C^{\widetilde Q}$, we denote by $\underline{X}$ and  $\underline{\mathcal X}$  the corresponding object and  subcategory of $\underline{\mathcal C^{\widetilde Q}}$ respectively. Since the action of the group $\Gamma$ on $\mathcal C^{\widetilde Q}$ is exact, $\underline{\mathcal C^{\widetilde Q}}$ also admits an exact $\Gamma$-action.

Since $\mathcal C^{\widetilde Q}$ is a Frobenius category, by the standard result, see \cite{BIRTS}, we have that

\begin{lemma}\label{extpre}
$Ext^1_{\mathcal C^{\widetilde Q}}(Z_1,Z_2)\cong Ext^1_{\underline{\mathcal C^{\widetilde Q}}}(\underline{Z_1},\underline{Z_2})$ for all $Z_1,Z_2\in \mathcal C^{\widetilde Q}$.
\end{lemma}

The category $\underline{\mathcal C^{\widetilde Q}}$ can be viewed as an additive categorification of the cluster algebra $\mathcal A(\widetilde\Sigma)$ given from $\widetilde Q$. For details, refer to \cite{FK} and \cite{PP1}. Although the authors of \cite{FK} and \cite{PP1} deal with the cluster algebras of  finite ranks, it is easy to see that these results still hold in $\underline{\mathcal C^{\widetilde Q}}$ since  $Q''$, as well as  $\widetilde Q$, is a strongly almost finite quiver.

Denote $\underline{\mathcal T_0}=add(\underline{\mathcal T'}\cup \underline{\mathcal T''})$, where $\underline{\mathcal T'}$ and $\underline{\mathcal T''}$ respectively consist of the indecomposable objects correspondent to cluster variables in the clusters $\widetilde X$ and $\widetilde Y$ of $\widetilde \Sigma$.

 Let $\mathcal U$ be the subcategory of $\underline{\mathcal C^{\widetilde Q}}$ generated by $\{\underline X\in \underline{\mathcal C^{\widetilde Q}}\;|\;Hom_{\underline{\mathcal C^{\widetilde Q}}}(\underline {T}[-1],\underline X)=0, \;\forall\; \underline T\in \underline{\mathcal T''}\}.$

For any $\underline{X}\in \mathcal U $, let $\underline{T}_1\rightarrow \underline{T}_0\overset{f}{\rightarrow}\underline{X}\rightarrow \underline{T}_1[1]$ be the triangle with $f$ the minimal right $\underline{\mathcal T_0}$-approximation.  By Lemma \ref{extpre}, $\underline{\mathcal T_0}$ is a cluster tilting subcategory of $\underline{\mathcal C^{\widetilde Q}}$. Applying $Hom_{\underline{\mathcal C^{\widetilde Q}}}(\underline{T},-)$ to the triangle for all $\underline{T}\in \underline{\mathcal T_0}$, we have $\underline{T}_1\in \underline{\mathcal T_0}$. The index of $\underline X$ is defined as $$ind_{\underline{\mathcal T_0}}(\underline{X})=[\underline{T}_0]-[\underline{T}_1]\in K_0(\underline{\mathcal T_0})\cong \mathbb Z^{|\widetilde Q_0|}.$$

Recall that the Gabriel quiver of $\underline{\mathcal T_0}$ is isomorphic to $\widetilde Q$. We may assume that $\{X_i\;|\;i\in \widetilde Q_0\}$ is the complete set of the indecomposable objects of $\underline{\mathcal T_0}$. For any $L\in mod\underline{\mathcal T_0}$, we denote $(dim_k (L(X_i)))_{i\in \widetilde Q_0}\in \bigoplus\limits_{i\in \widetilde Q_0}\mathbb Z^{e_i}$ as its dimensional vector.

After the preparations, we give the definition of cluster character $CC(\;\;)$ on $\underline{\mathcal C^{\widetilde Q}}$. For any $i\in \widetilde Q_0\setminus F$, since $\widetilde Q$ is strongly almost finite, we can set $$\widehat{y}_i=\prod\limits_{j\in \widetilde Q_0\setminus F} x_j^{b_{ji}}\prod\limits_{j'\in F}y_{j'}^{b_{j'i}}\in \mathbb Q[x_i^{\pm 1}, y_j\;|\;i\in \widetilde Q_0\setminus F,j\in F].$$ For each rigid object $\underline X\in \mathcal U$, we define

$$CC(\underline X)=\widetilde{\mathbf{x}}^{ind_{\underline{\mathcal T_0}}(\underline X)}\sum\limits_{\mathbf{a}\in\bigoplus\limits_{i\in \widetilde Q_0}\mathbb Ze_i}\chi(Gr_{\mathbf{a}}(Hom_{\underline{\mathcal C^{\widetilde Q}}}(-,\underline{X}[1])))\prod\limits_{j\in \widetilde Q_0\setminus F}\widehat{y_j}^{a_j},$$ where $\widetilde{\mathbf{x}}^{\mathbf{a}}=\Pi x^{a_i}_{i}$ for $\mathbf{a}=(a_i)_{i\in \widetilde Q_0}\in \bigoplus\limits_{i\in \widetilde Q_0}\mathbb Ze_i$, $Gr_{\mathbf{a}}(Hom_{\underline{\mathcal C}}(-,\underline{X}[1]))$ is the quiver Grassmannian whose points are corresponding to the sub-$\underline{\mathcal T_0}$-representations of $Hom_{\underline{\mathcal C^{\widetilde Q}}}(-,\underline{X}[1])$ with dimension vector $\mathbf{a}$, and $\chi$ is the Euler characteristic with respect to \'{e}tale cohomology with proper support.  It is easy to see that $CC(\underline X)\in \mathbb Q[x_i^{\pm 1}, y_j\;|\;i\in \widetilde Q_0\setminus F,j\in F]$.

\begin{theorem}\label{cluster ch}
Keeps the notations as above. Then

(1)  $CC(\underline {T_i})=x_i$ for all $\underline {T_i}\in \underline{\mathcal T'}$.

(2) $CC(\underline X\oplus \underline X')=CC(\underline X)CC(\underline X')$ for any objects $\underline X,\underline X'\in \mathcal U$.

(3)  $CC(\underline X)CC(\underline Y)=CC(\underline Z)+CC(\underline Z')$ for $\underline X,\underline Y\in \mathcal U$ with $dimExt^1_{\underline{\mathcal C^{\widetilde Q}}}(\underline X,\underline Y)=1$ and the two non-splitting triangles:\;
$ \underline Y\rightarrow \underline Z \rightarrow \underline X \rightarrow \underline Y[1]$ \;and\; $\underline X\rightarrow \underline Z' \rightarrow \underline Y \rightarrow \underline X[1].$
\end{theorem}

\begin{proof} This theorem can be proved similarly as that of (\cite{FK}, Theorem 3.3) using of local finiteness of $\widetilde Q$ since it is strongly almost finite.
\end{proof}

Like \cite{D}, for any $\underline X\in \mathcal U$, we define
$ind'_{\underline{\mathcal T_0}}(\underline{X})=\lambda(ind_{\underline{\mathcal T_0}}(\underline{X}))\in \bigoplus\limits_{[i]\in \overline{\widetilde Q_0}}\mathbb Ze_{[i]}.$

For each rigid object $\underline X\in \mathcal U$, using the definition of $\pi$ in (\ref{pi}), we define the {\bf $\Gamma$-equivariant cluster character} as follows:
$$CC'(\underline X)=\pi(CC(\underline X))=\mathbf{x}^{ind'_{\underline{\mathcal T_0}}\underline X}\sum\limits_{\mathbf{a}\in\bigoplus\limits_{i\in \widetilde Q_0}\mathbb Ze_i}\chi(Gr_{\mathbf{a}}Hom_{\underline{\mathcal C^{\widetilde Q}}}(-,\underline{X}[1]))\prod\limits_{[j]\in \overline{\widetilde Q_0}\setminus \overline F}\widehat{y}_{[j]}^{\lambda(\mathbf a)_{[j]}},$$ where $\mathbf{x}^{\mathbf{a}}=\Pi x^{a_{[i]}}_{[i]}$ for $\mathbf{a}=(a_{[i]})_{[i]\in \overline{\widetilde Q_0}\setminus \overline{F}}\in \bigoplus\limits_{[i]\in \overline{\widetilde Q_0}\setminus \overline{F}}\mathbb Ze_{[i]}$ and $\widehat y_{[i]}=\pi(\widehat y_i)=\prod\limits_{[j]\in \overline{\widetilde Q_0}\setminus \overline{F}} x_{[j]}^{b_{[j][i]}}\prod\limits_{[j']\in \overline{\widetilde Q_0}\setminus \overline{F}}y_{[j']}^{b_{[j'][i]}}.$

Inspired by Definition 3.34 of \cite{D}, we give the following definition,

\begin{definition}
 Two objects $\underline X, \underline X'\in \underline{\mathcal C^{\widetilde Q}}$ are said to be  {\bf equivalent modulo} $\Gamma$ if $\underline X\cong \bigoplus\limits_{k=1}^m\underline X_k,\; \underline X'\cong \bigoplus\limits_{k=1}^m\underline X'_k$ with $add\{h\cdot \underline X_k\;|\;h\in\Gamma\}= add\{h\cdot \underline X'_k\;|\;h\in\Gamma\}$ for every $k$ and indecomposables $\underline X_k,\underline X'_k$.
\end{definition}

Similar to Lemma 3.49 of \cite{D}, we have the following lemma,

\begin{lemma}
The Laurent polynomial $CC'(\underline X)$ depends only on the class of $\underline X$ under equivalent modulo $\Gamma$ for and $\underline X\in \underline{\mathcal C^{\widetilde Q}}$.
\end{lemma}

\begin{proof}
We need only to prove that $CC'(\underline X)=CC'(h\cdot \underline X)$ for $h\in \Gamma$ by Theorem \ref{cluster ch}(2). It follows immediately from that $ind'_{\underline{\mathcal T_0}}(\underline X)=ind'_{\underline{\mathcal T_0}}(h\cdot \underline X)$ and $\chi(Gr_{\mathbf{a}}Hom_{\underline{\mathcal C^{\widetilde Q}}}(-,\underline{X}[1]))=\chi(Gr_{h\cdot \mathbf{a}}Hom_{\underline{\mathcal C^{\widetilde Q}}}(-,h\cdot\underline{X}[1]))$.
\end{proof}

Using the algebra homomorphism $\pi$ and Theorem \ref{cluster ch}, we have the following theorem at once,

\begin{theorem}\label{char}
Keeps the notations as above. Then

(1)  $CC'(\underline {T_i})=x_{[i]}$ for all $\underline T_i\in \underline{\mathcal T'}$.

(2)  $CC'(\underline X\oplus \underline X')=CC'(\underline X)CC'(\underline X')$ for any two objects $\underline X,\underline X'\in \mathcal U$.

(3)  $CC'(\underline X)CC'(\underline Y)=CC'(\underline Z)+CC'(\underline Z')$ for $\underline X,\underline Y\in \mathcal U$ with  $dimExt^1_{\underline{\mathcal C^{\widetilde Q}}}(\underline X,\underline Y)=1$ satisfying two non-splitting triangles: \;
$ \underline Y\rightarrow \underline Z \rightarrow \underline X \rightarrow \underline Y[1]$ \;and\; $\underline X\rightarrow \underline Z' \rightarrow \underline Y \rightarrow \underline X[1].$
\end{theorem}

Following \cite{PP}, we say $\underline X\in \underline{\mathcal C^{\widetilde Q}}$ to be {\bf reachable} if it belongs to a cluster-tilting subcategory which can be obtained by a sequence of mutations from $\underline {\mathcal T_0}$ with the mutations do not take at the objects in $\underline{\mathcal T''}$. It is clear that any reachable object belongs to $\mathcal U$.

Following Theorem 4.1 of \cite{PP}, we have the following result:
\begin{theorem}\label{reach}
Keep the notations as above. Then the cluster character $CC'(\;\;)$ gives a surjection from the set of equivalence classes of indecomposable reachable objects under equivalent modulo $\Gamma$ of $\underline{\mathcal C^{\widetilde Q}}$ to the set of clusters variables of the cluster algebra $\mathcal A(\Sigma)$.
\end{theorem}

\begin{proof}
By Theorem \ref{char}, the proof is similar as that of Theorem 4.1 in \cite{PP}.
\end{proof}

\section{Sign-coherence of ${\bf  g}$-vectors}\label{sign-coherence}

Keep the notations in Section \ref{3}. In this section, we will prove the sign-coherence of ${\bf g}$-vectors for acyclic sign-skew-symmetric cluster algebras. For convenience, suppose that $\mathcal A(\Sigma)$ is an acyclic sign-skew-symmetric cluster algebra with principal coefficients at $\Sigma$.

Since $B$ is acyclic, $\left(\begin{array}{c}
B  \\
I_n
\end{array}\right)$ is acyclic, too. We can construct an unfolding $(\widetilde Q, F, \Gamma)$ according to Theorem \ref{mainlemma}. It is easy to check that the corresponding seed $\widetilde \Sigma$ of $(\widetilde Q, F, \Gamma)$ is with principal coefficients, where $\widetilde \Sigma$ is the seed associate to $(Q,F)$.

\begin{proposition} Assume that $\underline X$ is an object of $\mathcal U$ given in Section 3.

(i) \label{gindex}
 $CC(\underline X)$ admits a ${\bf g}$-vector $(g_i)_{i\in \widetilde Q_0\setminus F}$ which is given by $g_i=[ind_{\underline{\mathcal T_0}}(\underline X): \underline {T_i}]$ for each $i$.

(ii) \label{g-vec}
 $CC'(\underline X)$ admits a ${\bf g}$-vector $(g_{[i]})_{[i]\in \overline{Q_0}\setminus \overline{F}}$ which is given by $g_{[i]}=\sum\limits_{i'\in [i]}[ind_{\underline{\mathcal T_0}}(\underline X): \underline {T_{i'}}]$ for each $i$.
\end{proposition}

\begin{proof} (i)
Since the quiver $\widetilde Q$ is strongly almost finite, the proof is the same one as that of (\cite{PP1}, Proposition 3.6) using of the local finiteness of $\widetilde Q$.

(ii) is obtained immediately from (i).
\end{proof}

 Let $h\in \Gamma$ be either of finite order or without fixed points. Define  $\underline{\mathcal C^{\widetilde Q}}_h$ the $K$-linear category whose objects  are the same as that of $\underline{\mathcal C^{\widetilde Q}}$, and whose morphisms consist of  $Hom_{\underline{\mathcal C^{\widetilde Q}}_h}(\underline{X},\underline{Y})=\bigoplus\limits_{h'\in \Gamma'}Hom_{\mathcal C^{\widetilde Q}}(h'\cdot \underline{X},\underline{Y})$ for all objects $\underline{X},\underline{Y}$. We view this category as a dual construction of the category $\mathcal C^Q_h$ in \cite{HL}.

 Denote by $\underline{\mathcal C^{\widetilde Q}}_h(\underline{\mathcal T_0})$ the subcategory of $\underline{\mathcal C^{\widetilde Q}}_h$ consisting of all objects $\underline{T}\in \underline{\mathcal T_0}$.

\begin{lemma}\label{Hom-finite} The category  $\underline{\mathcal C^{\widetilde Q}}_h$ is $Hom$-finite if either (i)~ the order of $h$ is finite, or (ii) ~$Q$ has no fixed points  under the action of $h\in \Gamma$.
\end{lemma}

\begin{proof}
The proof is similar to that of Lemma 6.1 in \cite{HL}.
\end{proof}

In the sequel, assume that $\underline{\mathcal C^{\widetilde Q}}_h(\underline{\mathcal T_0})$ is $Hom$-finite. Let $F:\underline{\mathcal C^{\widetilde Q}}\rightarrow mod\underline{\mathcal C^{\widetilde Q}}_h(\underline{\mathcal T_0})^{op}$ be the functor mapping an object $X$ to the restricting of $\underline{\mathcal C^{\widetilde Q}}_h(-,X)$ to $\underline{\mathcal C^{\widetilde Q}}_h(\underline{\mathcal T_0})$. For each indecomposable object $\underline{T}$ of $\underline{\mathcal T_0}$, denote by $S_{\underline T}$  the {\bf simple quotient of $F(\underline T)$} via $S_{\underline T}(\underline T')=End_{\mathcal C}(\underline T)/J$ if $\underline T'\cong \underline T$ and $S_{\underline T}(\underline T')=0$ if $\underline T'\not\cong \underline T$, for any object $\underline T'$ of $\underline{\mathcal C^{\widetilde Q}}_h$, where $J$ is the Jacobson radical of $End_{\underline{\mathcal C^{\widetilde Q}}_h}(\underline T)$.

\begin{lemma}\label{lift}
Keep the notations as above. For any morphism $\widetilde f:F(\underline M)\rightarrow F(\underline N)$ in $mod\underline{\mathcal C^{\widetilde Q}}_h(\underline{\mathcal T_0})^{op}$ with $M,N\in \mathcal C^{\widetilde Q}$, there exists $f: \underline M\rightarrow \underline N$ in $\underline{\mathcal C^{\widetilde Q}}_h$ such that $F(f)=\widetilde f$.
\end{lemma}

\begin{proof}
Let $\underline{T}_1\overset{e}{\rightarrow} \underline{T}_0\overset{d}{\rightarrow}\underline{M}\rightarrow \underline{T}_1[1]$ be the triangle such that $d$ is a minimal right $\underline{\mathcal T_0}$-approximation. Since $\underline{\mathcal T_0}$ is a cluster tilting subcategory of $\underline{\mathcal C^{\widetilde Q}}$, we have $\underline{T}_1\in \underline{\mathcal T_0}$. Applying $F$ to the above triangle, we have $F(\underline{T}_1)\rightarrow F(\underline{T}_0)\rightarrow F(\underline{M})\rightarrow 0$. Similarly, there is a triangle $\underline{T'}_1\overset{e'}{\rightarrow} \underline{T'}_0\overset{d'}{\rightarrow}\underline{N}\rightarrow \underline{T'}_1[1]$, and $F(\underline{T'}_1)\rightarrow F(\underline{T'}_0)\rightarrow F(\underline{N})\rightarrow 0$. Since $\underline{T}_0,\underline{T'}_0,\underline{T}_1,\underline{T'}_1\in \underline{\mathcal T_0}$, it follows that  $F(\underline{T}_0),F(\underline{T'}_0),F(\underline{T}_1),F(\underline{T'}_1)$ are projective in $mod\underline{\mathcal C^{\widetilde Q}}_h(\underline{\mathcal T_0})^{op}$. Thus, $\widetilde f$ can be lift to the following commutative diagram:
$$\xymatrix{
F(\underline{T}_1)\ar[r]\ar[d]^{\widetilde f_1}   & F(\underline{T}_0)\ar[r] \ar[d]^{\widetilde f_0}     & F(\underline{M})\ar[r]\ar[d]_{\widetilde f}& 0 \\
F(\underline{T'}_1)\ar[r]                & F(\underline{T'}_0)\ar[r]      & F(\underline{N}) \ar[r] &0,}$$

By the Yoneda Lemma, there exist $f_1:\underline{T}_1\rightarrow \underline{T'}_1$ and $f_0:\underline{T}_0\rightarrow \underline{T'}_0$ in $\underline{\mathcal C^{\widetilde Q}}_h$ such that $F(f_1)=\widetilde f_1$ and $F(f_0)=\widetilde f_0$. Thus, $f_1\in\bigoplus\limits_{h'\in \Gamma'}Hom_{\underline{\mathcal C^{\widetilde Q}}}(h'\cdot \underline {T}_1,\underline {T'}_1)$ and $f_0\in\bigoplus\limits_{h'\in \Gamma'}Hom_{\underline{\mathcal C^{\widetilde Q}}}(h'\cdot \underline {T}_0,\underline {T'}_0)$ such that $$\xymatrix{
\bigoplus\limits_{h'\in \Gamma'}h'\cdot \underline{T}_1\ar[r]\ar[d]^{f_1}   & \bigoplus\limits_{h'\in \Gamma'}h'\cdot \underline{T}_0\ar[r] \ar[d]^{ f_0}     & \bigoplus\limits_{h'\in \Gamma'}h'\cdot \underline{M}\ar[r]\ar[d]& 0 \\
\underline{T'}_1\ar[r]                & \underline{T'}_0\ar[r]      & \underline{N} \ar[r] &0,}$$ commutes. Since $\underline{\mathcal C^{\widetilde Q}}_h(\underline{\mathcal T_0})$ is $Hom$-finite, we may assume that only finite $h'\in \Gamma'$ appear in the upper triangle, which means that there exists a finite subset $I$ of $\Gamma'$ such that $$\xymatrix{
\bigoplus\limits_{h'\in I}h'\cdot \underline{T}_1\ar[r]\ar[d]^{f_1}   & \bigoplus\limits_{h'\in I}h'\cdot \underline{T}_0\ar[r] \ar[d]^{ f_0}     & \bigoplus\limits_{h'\in I}h'\cdot \underline{M}\ar[r]\ar[d]& 0 \\
\underline{T'}_1\ar[r]                & \underline{T'}_0\ar[r]      & \underline{N} \ar[r] &0,}$$ commutes. Thus, there exists $f\in\bigoplus\limits_{h'\in I}Hom_{\underline{\mathcal C^{\widetilde Q}}}(h'\cdot \underline {T}_0,\underline {T'}_0)\subseteq \bigoplus\limits_{h'\in \Gamma'}Hom_{\underline{\mathcal C^{\widetilde Q}}}(h'\cdot \underline {T}_0,\underline {T'}_0)$ such that the above diagram commutes. This commutative diagram induces
$$\xymatrix{
F(\underline{T}_1)\ar[r]\ar[d]^{\widetilde f_1}   & F(\underline{T}_0)\ar[r] \ar[d]^{\widetilde f_0}     & F(\underline{M})\ar[r]\ar[d]_{F(f)}& 0 \\
F(\underline{T'}_1)\ar[r]                & F(\underline{T'}_0)\ar[r]      & F(\underline{N}) \ar[r] &0.}$$
Therefore, we have $F(f)=\widetilde f$.
\end{proof}

\begin{lemma}\label{minimalpro}
Assume that for $\underline X\in \mathcal U$, the category $add(\{h\cdot \underline{X}\;|\;h\in \Gamma\})$ is rigid. Let $\underline{T}_1\overset{f'}{\rightarrow} \underline{T}_0\overset{f}{\rightarrow}\underline{X}\rightarrow \underline{T}_1[1]$ be a triangle in $\underline{\mathcal C^{\widetilde Q}}$ with $f$ a minimal right $\underline{\mathcal T_0}$-approximation. If $\underline X$ has not any  direct summand in $\underline{\mathcal T_0}[1]$, then $$F(\underline{T}_1)\overset{F(f')}{\longrightarrow} F(\underline{T}_0)\overset{F(f)}{\longrightarrow}F(\underline{X})\rightarrow 0$$ is a minimal projective resolution of $F(\underline X)$.
\end{lemma}

\begin{proof}
Since $\underline X$ does not have direct summand in $\underline{\mathcal T_0}[1]$, $f'$ is right minimal. Otherwise, if $f'$ is not right minimal, then $f'$ has a direct summand as $\underline{T'}\rightarrow 0$, thus $\underline{T'}[1]$ is a direct summand of $\underline X$. This is a contradiction.

First, we prove that $F(f)$ is a projective cover of $F(\underline X)$. For any projective representation $F(\underline T)e$ of $\underline{\mathcal C^{\widetilde Q}}_h$ and surjective morphism $u:F(\underline T)e\rightarrow F(\underline X)\rightarrow 0$, where $\underline T\in \underline{\mathcal C^{\widetilde Q}}_h$, $e\in End_{\underline{\mathcal C^{\widetilde Q}}_h}(T)$ is an idempotent. By Yoneda Lemma, there exists $g\in Hom_{\underline{\mathcal C^{\widetilde Q}}_h}(\underline T,\underline X)$ such that $F(g)=u$. Since $F(\underline T)e$ and $F(\underline{T}_0)$ are projective, there exist $v: F(\underline{T}_0)\rightarrow F(\underline T)e$ and $w: F(\underline{T})\rightarrow F(\underline{T}_0)$ such that $F(f)=F(g)v$ and $F(g)=F(f)w$. Thus, $F(f)=F(f)wv$. Similarly, by Yoneda Lemma, there exist $g'\in Hom_{\underline{\mathcal C^{\widetilde Q}}_h}(\underline{T}_0,\underline T)$ and $g''\in Hom_{\underline{\mathcal C^{\widetilde Q}}_h}(\underline{T},\underline{T}_0)$ such that $F(g')=v$, $F(g'')=w$. Thus, $F(f)=F(f)wv$ is equivalent to $f=f\circ g''\circ g'$. We may assume $g''\circ g'=(g_{h'})_{h'\in \Gamma'}$, then $f=f\circ(g''\circ g')=(fg_{h'})_{h'\in \Gamma}$. Thus $f=fg_{e}$ and $0=fg_{h'}=0$ for any $h'\neq e$, where $e$ is the identity of $\Gamma'$. Further, since $f$ is a right minimal $\underline{\mathcal T_0}$-approximation and $h'\underline{T}_0\in \underline{\mathcal T_0}$~. Therefore, for any $e\neq h'\in \Gamma$, $g_{h'}\in J(h'\underline{T}_0,\underline{T}_0)$ and $g_e$ is an isomorphism. Using Lemma 5.10 of \cite{HL}, $g''\circ g'=(g_{h'})_{h'\in \Gamma'}$ is an isomorphism. Thus, $F(\underline{T}_0)$ is a direct summand of $F(\underline T)e$.

Similarly, because $f'$ is right minimal, $F(f')$ induces a projective cover of $ker(F(f))$. Our result follows.
\end{proof}

 Inspired by (Proposition 2.1, \cite{DK}), (Lemma 3.58, \cite{D}) and (Lemma 3.5, \cite{PP1}), we have:

\begin{lemma}\label{distinct}
Assume $\underline X\in \mathcal U$ such that $add(\{h\cdot \underline{X}\;|\;h\in \Gamma\})$ is rigid, and $\underline{T}_1\overset{f'}{\rightarrow} \underline{T}_0\overset{f}{\rightarrow}\underline{X}\rightarrow \underline{T}_1[1]$ is a triangle in $\underline{\mathcal C^{\widetilde Q}}$ and $f$ is a minimal right $\underline{\mathcal T}$-approximation. If $\underline{T}$ is a direct summand of $\underline{T}_0$ for indecomposable object $\underline T\in \underline{\mathcal C^{\widetilde Q}}$, then $h\cdot \underline{T}$ is not a direct summand of $\underline{T}_1$ for any $h\in \Gamma$.
\end{lemma}

\begin{proof}
By Lemma 2.3 and Lemma 2.4 of \cite{HL}, we may assume that $h$ has no fixed points or finite order. By Lemma \ref{Hom-finite}, $\underline{\mathcal C^{\widetilde Q}}_h$ is Hom-finite. For any $h'\in \Gamma'$ and $\underline T'\in \underline{\mathcal T}$, applying $Hom_{\underline{\mathcal C^{\widetilde Q}}}(h'\cdot \underline T',-)$ to the triangle, we get
$$Hom_{\underline{\mathcal C^{\widetilde Q}}}(h'\cdot \underline T',\underline{T}_1)\rightarrow Hom_{\underline{\mathcal C^{\widetilde Q}}}(h'\cdot \underline T',\underline{T}_0)\overset{Hom_{\underline{\mathcal C^{\widetilde Q}}}(h'\cdot \underline T',f)}{\longrightarrow}Hom_{\underline{\mathcal C^{\widetilde Q}}}(h'\cdot \underline T',\underline{X})\rightarrow 0.$$
Since $f$ is minimal, we get a minimal projective resolution of $F(\underline X)$,
$$F(\underline{T}_1)\rightarrow F(\underline{T}_0)\rightarrow F(\underline X)\rightarrow 0.$$ To prove that $h\cdot \underline T$ is not a direct summand of $\underline{T}_1$, it suffices to prove that $F(\underline T)$ is not a direct summand of $F(\underline{T}_1)$, or equivalently $Ext^1(F(\underline{X}),S_{\underline T})=0$, where $S_{\underline T}$ is the simple quotient of $F(\underline T)$.

As $F(\underline{T}_0)\rightarrow F(\underline X)$ is the projective cover of $F(\underline X)$ and $F(\underline{T})$ is a direct summand of $F(\underline{T}_0)$, then there is a non-zero morphism $\widetilde p:F(\underline X)\rightarrow S_{\underline T}$. For any $\widetilde g:F(\underline T_1)\rightarrow S_{\underline T}$, since $F(\underline T_1)$ is projective, there exists $\widetilde q: F(\underline T_1)\rightarrow F(\underline X)$ such that $\widetilde g=\widetilde p \widetilde q$.

Let $T$ be a lifting of $\underline T$, by Lemma \ref{extpre}, we have $T\in \mathcal T$ and $T$ is non-projective. Moreover, since $\underline T$ is indecomposable, we can choose $T$ is indecomposable. By Lemma 5.3 of \cite{HL}, there is an admissible short exact sequence $0\rightarrow Y\rightarrow T'\rightarrow T\rightarrow 0$. Since $\mathcal T$ has no $\Gamma$-loop, by the dual version of Lemma 5.11 (2) of \cite{HL}, and Lemma \ref{extpre}, we have $S_{\underline T}\cong F(\underline Y[1])$.

Thus, according to Lemma \ref{lift}, lifting $\widetilde q$, $\widetilde g$ and $\widetilde p$ as $q,g,p$ in $\underline{\mathcal C^{\widetilde Q}}_h$, where $q\in\bigoplus\limits_{h'\in \Gamma'}Hom_{\underline{\mathcal C^{Q}}}(h'\cdot \underline T_1,\underline X)$, $g\in\bigoplus\limits_{h'\in \Gamma'}Hom_{\underline{\mathcal C^{\widetilde Q}}}(h'\cdot \underline T_1,\underline Y[1])$ and $p\in\bigoplus\limits_{h'\in \Gamma'}Hom_{\underline{\mathcal C^{\widetilde Q}}}(h'\cdot \underline X,\underline Y[1])$. Since $\widetilde g=\widetilde p \widetilde q$ and $Hom(F(\underline T),F(\underline Y[1]))\cong Hom_{\underline{\mathcal C^{\widetilde Q}}_h}(\underline T,\underline Y[1])$, we obtain $g=p\circ q$.

Since $\underline{\mathcal C^{\widetilde Q}}_h$ is $Hom$-finite, we may assume that $g\in\bigoplus\limits_{h'\in I}Hom_{\underline{\mathcal C^{\widetilde Q}}}(h'\cdot \underline T_1,\underline Y[1])$ and $q\in\bigoplus\limits_{h'\in I}Hom_{\underline{\mathcal C^{\widetilde Q}}}(h'\cdot \underline T_1,\underline X)$ for a finite subset $I\subseteq \Gamma'$. According to the composition of morphisms in $\underline{\mathcal C^{\widetilde Q}}_h$, $g=p\circ q$ means $g=p(\sum\limits_{h'\in I}h'\cdot q)$, equivalently, we have the following commutative diagram:
$$\xymatrix{
        \bigoplus\limits_{h'\in I}h'\cdot \underline{X}[1] \ar[r]^{\sum h'\cdot f'} &\bigoplus\limits_{h'\in I}h'\cdot \underline{T}_1\ar[r]^{\sum h'\cdot f}\ar[ld]_{\sum h'\cdot q}\ar[d]^{g}  & \bigoplus\limits_{h'\in I}h'\cdot \underline{T}_0\ar[r] &\bigoplus\limits_{h'\in I}h'\cdot \underline{X} \\
        \bigoplus\limits_{h'\in I}h'\cdot \underline{X}      \ar[r]^{p} &\underline{Y}[1],            &       &       &}$$
 Since $add(\{h\cdot X\;|\;h\in \Gamma\})$ is rigid, we have $g(\sum h'\cdot f')=0$. Therefore, $g$ factor through $\sum h'\cdot f$. Thus, in $\underline{\mathcal C^{\widetilde Q}}_h$, $g$ factors through $f$. By the arbitrary of $g$, we get a surjective map $Hom_{\underline{\mathcal C^{\widetilde Q}}_h}(\underline T_0,\underline Y[1])\twoheadrightarrow Hom_{\underline{\mathcal C^{\widetilde Q}}_h}(\underline T_1,\underline Y[1])$. Since $Hom_{\underline{\mathcal C^{\widetilde Q}}_h}(\underline T_i,\underline Y[1])\cong Hom(F(\underline T_i), F(\underline Y[1]))=Hom(F(\underline T_i), S_{\underline T})$ for $i=1,2$, we get $Hom(F(\underline T_0), S_{\underline T})\rightarrow Hom(F(\underline T_0), S_{\underline T})$ is surjective. Therefore, we obtain $Ext^1(F(\underline X), S_{\underline T})=0$. Our result follows.
\end{proof}

\begin{corollary}
Let $\underline{X}$ be an object of $\underline{\mathcal C}$ such that $add(\{h\cdot \underline{X}\;|\;h\in \Gamma\})$ is rigid. If $ind'_{\underline{\mathcal T_0}}(\underline{X})$ has no negative coordinates, then $\underline{X}\in \underline{\mathcal T}$.
\end{corollary}

\begin{proof}
Assume $\underline{T}_1\overset{f'}{\rightarrow} \underline{T}_0\overset{f}{\rightarrow}\underline{X}\rightarrow \underline{T}_1[1]$ is a triangle and $f$ is a minimal right $\mathcal T$-approximation. According to Lemma \ref{distinct}, $\underline{T}_0$ and $\underline{T}_1$ have no direct summands which have the same $\Gamma$-orbits. Further, since $ind'_{\underline{\mathcal T}}(\underline{X})$ has no negative components. Therefore, by the definition of $ind'_{\underline{\mathcal T_0}}$, we obtain $\underline{T}_1=0$. Thus, $\underline{X}\cong\underline{T}_0\in \underline{\mathcal T}$.
\end{proof}

Using the above preparation,  we now can prove Conjecture \ref{$g$-vec} for all acyclic sign-skew-symmetric cluster algebras. The method of the proof follows from that of  Theorem 3.7 (i) in \cite{PP1}.

\begin{theorem}\label{sign-c}
The conjecture \ref{$g$-vec} on sign-coherence holds for all acyclic sign-skew-symmetric cluster algebras.
\end{theorem}

\begin{proof}
For any cluster $Z=\{z_1,\cdots,z_n\}$ of $\mathcal A(\Sigma)$, by Theorem \ref{reach}, we associate a cluster tilting subcategory $\underline{\mathcal T}$ of $\underline{\mathcal C^{\widetilde Q}}$ which obtained by a series of mutations of $\underline{\mathcal T_0}$ and the mutations do no take at $\underline{\mathcal T''}$. Precisely, there are $n$ indecomposable objects $\{\underline{X_j}\;|\;j=1,\cdots,n\}$ such that $\underline{\mathcal T}=add(\{h\cdot \underline{X_j}\;|\;j=1,\cdots,n\}\cup \underline{\mathcal T''})$ and $CC'(\underline{X_i})=z_i$ for $i=1,\cdots,n$. By Proposition \ref{g-vec}, the $g$-vector $g^j_{[1]},\cdots,g^j_{[n]}$ of $z_j$ is given by $g^j_{[i]}=\sum
\limits_{i'\in[i]}[ind_{\underline{\mathcal T_0}}(\underline{X_j}):\underline{T_{i'}}]$.

Suppose that there exist $s$ and $s'$ such that $g^s_{[i]}>0$ and $g^{s'}_{[i]}<0$. Assume that $\underline{T^j}_1\rightarrow \underline{T^j}_0\overset{f^j}{\rightarrow}\underline{X_j}\rightarrow \underline{T^j}_1[1]$ be the triangle with $f^j$ is a minimal right $\underline{\mathcal T_0}$-approximation for $j=s, s'$. Thus, there exist $h,h'\in \Gamma$ such that $h\cdot\underline{T_i}$ (respectively, $h'\cdot\underline{T_i}$) is a direct summand of $\underline{T^s}_0$ (respectively, $\underline{T^{s'}}_1$). Furthermore, in the triangle $$\bigoplus\limits_{j=s,s'}\underline{T^j}_1\rightarrow \bigoplus\limits_{j=s,s'}\underline{T^j}_0\overset{\bigoplus\limits_{j=s,s'}f^j}{\rightarrow}\bigoplus\limits_{j=s,s'}\underline{X_j}\rightarrow \bigoplus\limits_{j=s,s'}\underline{T^j}_1[1],$$ the morphism $\bigoplus\limits_{j=s,s'}f^j$ is a minimal right $\underline{\mathcal T_0}$-approximation. According to Lemma \ref{distinct}, $h'\cdot\underline{T_i}$ is not a direct summand of $\underline{T^{s'}}_1$ since $h\cdot\underline{T_i}$ is a direct summand of $\underline{T^{s}}_0$. This is a contradiction. Our result follows.
\end{proof}

\section{The recurrence of ${\bf g}$-vectors}\label{basechange1}

\begin{theorem}(\cite{DWZ})
Conjecture \ref{basechange} holds true for all finite rank skew-symmetric cluster algebras.
\end{theorem}

It is easy to see that the above theorem can be extended to the situation of infinite rank skew-symmetric cluster algebras, that is, we have:

\begin{theorem}\label{skew}
Conjecture \ref{basechange} holds true for all infinite rank skew-symmetric cluster algebras.
\end{theorem}

We first give the following easy lemma.

\begin{lemma}\label{finite}
Let $(Q,F,\Gamma)$ be the unfolding of a matrix $B$ and $\mathcal A=\mathcal A(\Sigma(Q,F))$. For any sequence of orbits $([i_1],\cdots,[i_s])$ and $a\in Q_0$, there exist finite subsets $S_j\subseteq [i_j], j=1,\cdots,s$ such that $\prod\limits_{k\in V_s}\mu_{k}\cdots\prod\limits_{k\in V_1}\mu_k(x_a)=\widetilde\mu_{[i_s]}\cdots\widetilde\mu_{[i_1]}(x_a)$ for all finite subsets $V_j, j=1,\cdots,s$, satisfying $S_j\subseteq V_j\subseteq [i_j], j=1,\cdots,s$.
\end{lemma}

\begin{proof}
  Since $\widetilde\mu_{[i_s]}\cdots\widetilde\mu_{[i_1]}(x_a)$ is determined by finite vertices of $Q_0$, there exist finite subsets $S_j\subseteq [i_j], j=1,\cdots,s$ such that $\prod\limits_{k\in S_s}\mu_{k}\cdots\prod\limits_{k\in S_1}\mu_k(x_a)=\widetilde\mu_{[i_s]}\cdots\widetilde\mu_{[i_1]}(x_a)$. Then for all finite subsets $V_j, j=1,\cdots,s$ satisfying $S_j\subseteq V_j\subseteq [i_j], j=1,\cdots,s$, we have $\prod\limits_{k\in V_s}\mu_{k}\cdots\prod\limits_{k\in V_1}\mu_k(x_a)=\widetilde\mu_{[i_s]}\cdots\widetilde\mu_{[i_1]}(x_a)$.
\end{proof}

Let $\mathcal A_1$ (respectively, $\mathcal A_2$) be the cluster algebra with principal coefficients at $\Sigma_1=(\widetilde X,\widetilde Y,Q)$ (respectively, $\Sigma_2=(\widetilde X', \widetilde Y', \widetilde \mu_{[k]}(Q))$). For any sequence $([i_1],\cdots,[i_s])$ of orbits of $Q_0$ and $a\in Q_0=\widetilde \mu_{[k]}(Q_0)$, denote $g^{Q,a}=(g^Q_i)_{i\in Q_0}$ (respectively, $g^{\widetilde \mu_{[k]}(Q),a}=(g^{\widetilde \mu_{[k]}(Q)}_i)_{i\in Q_0}$) be the ${\bf g}$-vector of the cluster variable $\widetilde\mu_{[i_s]}\cdots\widetilde\mu_{[i_1]}(x_a)$ (respectively, $\widetilde\mu_{[i_s]}\cdots\widetilde\mu_{[i_1]}\widetilde \mu_{[k]}(x'_a)$.

As a consequence of Theorem \ref{skew}, we have the following property.

\begin{proposition}\label{rec}
Keep the notations as above. The following recurrence holds:
\begin{equation}\label{eqn3.1}
g_i^{\widetilde \mu_{[k]}(Q)}=\begin{cases}-g_{i}^{Q}& \text{if }i\in [k];\\g_i^{Q}+\sum\limits_{k'\in [k]}[b_{ik'}]_+g_{k'}^{Q}-\sum\limits_{k'\in [k]}b_{ik'}min(g_{k'}^{Q},0)&\text{if }i\not\in [k].\end{cases}
\end{equation}
\end{proposition}

\begin{proof}
By Lemma \ref{finite}, for $j=1,\cdots,s$, there exist finite subsets $S^{1}_j\subseteq [i_j]$ (respectively, $S^{2}_j\subseteq [i_j]$) such that $g^{Q,a}$ (respectively, $g^{\widetilde \mu_{[k]}(Q),a}$) is the ${\bf g}$-vector of the cluster variable $\prod\limits_{t\in V^1_s}\mu_{t}\cdots\prod\limits_{t\in V^1_1}\mu_{t}(x_a)$ (respectively, $\prod\limits_{t\in V^2_s}\mu_{t}\cdots\prod\limits_{t\in V^1_1}\mu_{t}\widetilde\mu_{[k]}(x'_a)$) for all finite subsets $V^1_j$ (respectively, $V^2_j$) satisfying that $S^1_j\subseteq V^1_j\subseteq [i_j]$ (respectively, $S^2_j\subseteq V^2_j\subseteq [i_j]$). Choose $S_j=S^1_j\cup S^2_j$, we have $g^{Q,a}$ and $g^{\widetilde \mu_{[k]}(Q),a}$ as the ${\bf g}$-vectors of the cluster variables $\prod\limits_{t\in S_s}\mu_{t}\cdots\prod\limits_{t\in S_1}\mu_{t}(x_a)$ and $\prod\limits_{t\in S_s}\mu_{t}\cdots\prod\limits_{t\in S_1}\mu_{t}\widetilde\mu_{[k]}(x'_a)$ respectively.

For any finite subset $T\subseteq [k]$, denote by $\mathcal A^T$  the cluster algebra with principal coefficients at $\Sigma^T=(\widetilde X^T, \widetilde Y^T, \prod\limits_{k'\in T}\mu_{k'}(Q))$. Denote by $g^{\prod\limits_{k'\in T}\mu_{k'}(Q),a}=(g_{i}^{\prod\limits_{k'\in T}\mu_{k'}(Q)})_{i\in Q_0}$  the ${\bf g}$-vector of the cluster variable $\prod\limits_{t\in S_s}\mu_{t}\cdots\prod\limits_{t\in S_1}\mu_{t}\prod\limits_{k'\in T}\mu_{k'}(x_a^T)$.

Since the cluster variable $\prod\limits_{t\in S_s}\mu_{t}\cdots\prod\limits_{t\in S_1}\mu_{t}\widetilde\mu_{[k]}(x'_a)$ is only determined by finite vertices of $Q_0$, there exists a finite subset $S\subseteq [k]$ such that $g_i^{\widetilde \mu_{[k]}(Q)}=g_i^{\prod\limits_{k'\in T}\mu_{k'}(Q)}$ for any finite set $T$ satisfying $S\subset T\subseteq [k]$.

Furthermore, by Theorem \ref{skew}, for any subset $T\subseteq Q_0$, we have
\begin{equation}\label{eqn3.1}
g_i^{\prod\limits_{k'\in T}\mu_{k'}(Q)}=\begin{cases}-g_{i}^{Q}& \text{if }i\in T;\\g_i^{Q}+\sum\limits_{k'\in T}[b_{ik'}]_+g_{k'}^{Q}-\sum\limits_{k'\in T}b_{ik'}min(g_{k'}^{Q},0)&\text{if }i\not\in T.\end{cases}
\end{equation}

Therefore, the result holds.
\end{proof}

\begin{theorem}\label{base}
Conjecture \ref{basechange} holds true for all acyclic sign-skew-symmetric cluster algebras. That is,  let $B=(b_{[i][j]})\in Mat_{n\times n}(\mathbb Z)$ be a sign-skew-symmetric matrix which is mutation equivalent to an acyclic matrix and let $t_1\overset{[k]}{--} t_2\in \mathbb T_n$ and $B^2=\mu_{[k]}(B^1)$. For $[a]\in \{[1],\cdots,[n]\}$ and $t\in \mathbb T_n$, assume ${\bf g}^{B^1; t_1}_{[a]; t}=(g_{[1]}^{t_1},\cdots,g_{[n]}^{t_1})$ and ${\bf g}^{B^2; t_2}_{[a]; t}=(g_{[1]}^{t_2},\cdots,g_{[n]}^{t_2})$, then
\begin{equation}\label{eqn3.1}
g_{[i]}^{t_2}=\begin{cases}-g_{[k]}^{t_1}& \text{if }[i]=[k];\\g_{[i]}^{t_1}+[b_{[i][k]}^{t_1}]_+g_{[k]}^{t_1}-b_{[i][k]}^{t_1}min(g_{[k]}^{t_1},0)&\text{if }[i]\neq [k].\end{cases}
\end{equation}
\end{theorem}

\begin{proof}
By Lemma \ref{mut} and Theorem \ref{mainlemma}, let $(Q,\Gamma)$ be an unfolding of $B$. Assume $t_1 \overset{i_1}{--}t_2'\cdots t'_s\overset{i_s}{--} t$. By Theorem \ref{ho}, we have $g_{[i]}^{t_2}=\sum\limits_{i'\in [i]}g_{i'}^{\widetilde \mu_{[k]}(Q)}$ and $g_{[i]}^{t_1}=\sum\limits_{i'\in [i]}g_{i'}^{Q}$. By Proposition \ref{gindex} and Lemma \ref{distinct}, for $k',k''\in [k]$, both $g_{k'}^Q$ and $g_{k''}^Q$ are non-negative or non-positive. Thus, $\sum\limits_{k'\in [k]}min(g^Q_{k'},0)=min(\sum\limits_{k'\in [k]}g^Q_{k'},0)$. Using Proposition \ref{rec}, if $[i]=[k]$, then
$$g_{[i]}^{t_2}=\sum\limits_{i'\in [i]}g_{i'}^{\widetilde \mu_{[k]}(Q)}=-\sum\limits_{i'\in [i]}g_{i'}^{Q}=-g_{[i]}^{t_1};$$ if $[i]\neq [k]$, since $\sum\limits_{k'\in [k]}min(g^Q_{k'},0)=min(\sum\limits_{k'\in [k]}g^Q_{K'},0)$ and $b_{[i][k]}^{t_1}=\sum\limits_{i'\in [i]}b_{i'k}$, then
$$g_{[i]}^{t_2}=\sum\limits_{i'\in [i]}(g_{i'}^{Q}+\sum\limits_{k'\in [k]}[b_{i'k'}^{t_1}]_+g_{k'}^{Q}-\sum\limits_{k'\in [k]}b_{i'k'}^{t_1}min(g_{k'}^{Q},0))=g_{[i]}^{t_1}+[b_{[i][k]}^{t_1}]_+g_{[k]}^{t_1}-b_{[i][k]}^{t_1}min(g_{[k]}^{t_1},0).$$
The result holds.
\end{proof}

{\bf Acknowledgements:}\; This project is supported by the National Natural Science Foundation of China (No.11671350 and No.11571173).

\bibliographystyle{amsplain}

\end{document}